\documentclass[11pt]{amsart}

\usepackage[utf8]{inputenc}
\usepackage[T1]{fontenc}
\usepackage[english]{babel}
\usepackage{amsthm}

\usepackage{amsmath}
\usepackage{mathtools}
\usepackage{tikz}
\usetikzlibrary{matrix}
\usepackage{amsfonts}
\usepackage{amssymb}
\usepackage{graphicx}
\usepackage{xspace}
\usepackage{enumitem}
\usepackage{float}
\floatstyle{plaintop}
\restylefloat{table}
\usepackage{xfrac}
\usepackage[all]{xy}

\usepackage[mathscr]{eucal} 
\usepackage{amsmath} 
\usepackage{epsfig}
\usepackage{amscd}
\usepackage{verbatim}
\usepackage{booktabs}

\newtheorem{THEOR}{Theorem}[section]
\newtheorem{PROP}[THEOR]{Proposition}
\newtheorem{LEM}[THEOR]{Lemma}

\newtheorem{COR}[THEOR]{Corollary}
\newtheorem{REM}[THEOR]{Remark}

\theoremstyle{definition}

\numberwithin{equation}{section}

\renewcommand{\phi}{\varphi}

\newcommand{\lra}{\longrightarrow}
\newcommand{\ra}{\rightarrow}

\newcounter{example}[section]

\begin{document}

	\title{Second fundamental form and higher Gaussian maps}
	
	\dedicatory{Dedicated to the memory of Alberto Collino}

	 \author[Paola Frediani]{Paola Frediani} \address{ Dipartimento di
	Matematica, Universit\`a di Pavia, via Ferrata 5, I-27100 Pavia,
	 Italy } \email{{\tt paola.frediani@unipv.it}}

	\begin{abstract}
	In this paper we show a relation between higher even Gaussian maps of the canonical bundle on a smooth projective curve of genus $g \geq 4$ and the second fundamental form of the Torelli map. This generalises a result obtained by Colombo, Pirola and Tortora on the second Gaussian map and the second fundamental form. As a consequence, we prove that for any non-hyperelliptic curve, the Gaussian map $\mu_{6g-6}$ is injective,  hence all even Gaussian maps $\mu_{2k}$ are identically zero for all $k >3g-3$. We also give an estimate for the rank of $\mu_{2k}$ for $g-1 \leq k \leq 3g-3.$ 
		\end{abstract}

 \thanks{The author was partially supported MIUR PRIN 2017
``Moduli spaces and Lie Theory'' ,  by MIUR, Programma Dipartimenti di Eccellenza(2018-2022) - Dipartimento di Matematica ``F. Casorati'', Universit\`a degli Studi di Pavia and by INdAM (GNSAGA) }

	\date{}
	\maketitle

	\section{Introduction}
	
The purpose of this paper is to show a relation between the behaviour of even higher Gaussian maps $\mu_{2k}$ of the canonical bundle of a smooth projective curve of genus $g \geq 4 $ and the second fundamental form of the Torelli map. 

Denote by  $j: {\mathcal M}_g \rightarrow {\mathcal A}_g$ the Torelli map. It is an orbifold immersion outside the hyperelliptic locus.  The space ${\mathcal A}_g$ is the quotient of the Siegel space ${\mathcal H}_g$ by the action of the symplectic group $Sp(2g, {\mathbb Z})$, hence it inherits an orbifold metric, which is induced by the symmetric metric on the Hermitian symmetric  domain ${\mathcal H}_g$. 

One expects the Jacobian locus, that is the image of the Torelli map, to be rather curved with respect to the Siegel metric. One way to establish these metric properties of the Jacobian locus is to study the second fundamental form of the Torelli map on the complement of the hyperelliptic locus. This turns out to be quite difficult, also because the second fundamental form is non-holomorphic. Nevertheless the composition of the second fundamental form with a suitable projection turns out to be holomorphic. 
More precisely, in \cite{cpt} it was proven that the second fundamental form of the Torelli map at a non-hyperelliptic curve $[C]$ with respect to this metric is a lifting of the second Gaussian map $\mu_2$ of the canonical bundle on $C$, that varies holomorphically in the moduli space. This result has been very useful to study geometric properties of the Jacobian locus, namely to determine the holomorphic sectional curvature of ${\mathcal M}_g$ on the tangent directions given by Schiffer variations  (\cite{cf-trans}), and  to give bounds on the dimension of a germ of a totally geodesic subvariety of ${\mathcal A}_g$ contained in the Torelli locus (\cite{cfg}, \cite{fp}, \cite{gpt}). 

The second fundamental form of the Torelli map at a non-hyperelliptic curve $[C] \in {\mathcal M}_g$ is a map $ I_2(\omega_C) \rightarrow Sym^2 H^0(C, \omega_C^{\otimes 2})$, where $ I_2(\omega_C)$ is the vector space of quadrics containing the canonical image of $C$. In \cite{cpt} it is proven that its  composition  with the multiplication map $Sym^2 H^0(C, \omega_C^{\otimes 2}) \rightarrow H^0(C, \omega_C^{\otimes 4})$ is the second Gaussian map $\mu_2$. To prove  this equality, the authors first showed that the second fundamental form equals the so-called Hodge-Gaussian map $\rho$ and then,  $\forall Q \in I_2(\omega_C)$, they computed $\rho(Q)$ on the tensors $\xi_p \otimes \xi_p$, where $\xi_p \in H^1(C, T_C)$ is a Schiffer variation at a point $p \in C$.  Recall that $I_2(\omega_C) = ker (\mu_0)$, where $\mu_0: Sym^2H^0(C, \omega_C) \ra H^0(C, \omega_C^{\otimes 2})$ is the multiplication map and the second Gaussian map  $\mu_2$ is defined on $I_2(\omega_C) = ker(\mu_0)$ with values in $H^0 (C, \omega_C^{\otimes 4})$. 

In this paper we generalise the computation  of $\rho$ on quadrics in $Ker( \mu_0)$ done in \cite{cpt}, by  computing $\rho(Q)$ for quadrics contained in the kernel of higher even gaussian maps $\mu_{2k}$  on tensor products of higher Schiffer variations $\xi_p^k \otimes \xi_p^n$ at a point $p$. We notice that this is the first explicit computation of the second fundamental form on tangent directions different from the Schiffer variations $\xi_p$. We prove the following result (Theorem \ref{thmA}).

\begin{THEOR}
\label{thmA-intro}
Take $p$ a general point in $C$ and $1\leq n \leq 3g-3$. Then for every $Q \in ker(\mu_{2n-2})$, and for every $v, w \in \langle \xi^1_p,..., \xi^{n-1}_p \rangle$, we have 
\begin{enumerate}
\item $$\rho(Q)(v)(w) = \rho(Q)(v \odot w) =0.$$ 
\item $$\rho(Q)(\xi^{n}_{p} )(\xi^l_p)= \rho(Q)(\xi^{n}_{p} \odot \xi^l_p) = 0, \  \forall l \leq n-1.$$ 
\item If $Q \in  ker(\mu_{2n-2})$, and $Q \not\in ker(\mu_{2n})$, then 
$$\rho(Q)(\xi^{n}_p)(\xi^{n}_p)  = \rho(Q)(\xi^{n}_p \odot \xi^{n}_p) =  c_{n, 2n-2} \mu_{2n}(Q)(p) \neq 0.$$

\item For $n \geq 2$ the same result holds if $C$ is a general curve of genus $g$, $p \in C$ is any point and $n < \frac{\sqrt{g-2} }{4} -1$. 

\item For $n=1$ the same result  holds  for every curve $C$ of genus at least 5 which is non hyperelliptic and non trigonal and $\forall p \in C$. 
\end{enumerate}

\end{THEOR}

From this we deduce the following  (Theorem \ref{mui}).

\begin{THEOR}
\label{mui-intro}

Assume that for a curve $C$ of genus $g \geq 4$ and for some $n \leq 3g-3$, we have:

$ker(\mu_{2n}) \subsetneq ker(\mu_{2n-2}) \subsetneq  ker(\mu_{2n-4}) .... \subsetneq ker(\mu_{2})  \subsetneq ker(\mu_{0}) = I_2(K_C). $

Let $Y$ be a  germ of  a totally geodesic submanifold of ${\mathcal A}_g$ generically contained in ${\mathcal M}_g$ passing through $j(C)$. Then, for a generic $p \in C$, we have: 
\begin{enumerate}

\item $T_{(jC)}Y \cap \langle \xi^1_p, ..., \xi^n_p \rangle = \{0\}$. 

\item $\dim Y \leq 3g-3-n$. 
\end{enumerate}
\end{THEOR}

We recall that the second Gaussian  map $\mu_2$ is surjective for the general curve of genus at least $18$ (see \cite{ccm}). 
In \cite[Theorem D]{ro} it is shown that $\mu_{k}$ is  surjective for the general curve of genus $g$, $\forall k < \frac{\sqrt{g-2} }{2} -2.$ Hence the assumptions of the above Theorem are satisfied for a general curve of genus $g$ if $n < \frac{\sqrt{g-2} }{4} -1$. 

We also give an interpretation of the above result in terms of the geometry of the bicanonical image of $C$ (see Corollary \ref{bicanonical1}). 

It would be very interesting, given a curve $C$,  to know what is the maximal $n$ such that $\mu_{2k} \neq 0$, forall $k \leq n$. This would give information about totally geodesic subvarieties of ${\mathcal A}_g$ passing through $[j(C)]$ and possibly improve the known upper bounds on their dimension.  We plan to address this problem in later work.

Finally, using the fact that the second fundamental form is injective (\cite{cf-trans}, \cite{cfg}), we prove a result that only concerns the Gaussian maps $\mu_{2k}$. 

\begin{THEOR} (see Theorem \ref{rank}). 
For every non hyperelliptic  curve $C$ of genus $g \geq 4$, we have: 
\begin{enumerate}
\item $ker (\mu_{6g-6}) =0$, hence $\mu_{2k} \equiv 0$, $\forall k > 3g-3$. 
\item $\forall 0 \leq l \leq 3g-3$, $\dim (ker( \mu_{6g-6-2(l+1)} ))\leq (l+1)^2$, hence $rank(\mu_{6g-6-2l} ) \leq (l+1)^2$. 
\end{enumerate}

\end{THEOR}

Notice that the Gaussian map $\mu_{6g-6-2l}$ takes values in $H^0(C, \omega_C^{\otimes (6g-4-2l)})$, whose dimension is $(g-1)(12g-9-4l)$. So the estimate in $(2) $ in the above Theorem gives information only when $(l+1)^2 \leq (g-1)(12g-9-4l)$, so $l \leq 1-2g + \sqrt{(g-1)(16g-9)} \leq 2g-2$, and therefore $6g-6-2l \geq 2g-2$.

The structure of the paper is as follows. 
In section 2 we first define higher Schiffer variations and their relation with the bicanonical curve. Then we recall the definition of Hodge-Gaussian maps introduced in \cite{cpt} and their relation with the second fundamental form of the period map. Finally we recall the definition of Gaussian maps. 
In section 3 we prove the main results, namely the computation of $\rho(Q)$ on higher Schiffer variations for $Q \in ker (\mu_{2k})$, Theorem \ref{thmA} and Theorem \ref{rank}. Finally we show some consequences of these results.

\section*{Acknowledgements} 
I would like to thank Alessandro Ghigi for several useful conversations and suggestions.

	\section{Preliminaries}
	\subsection{Higher Schiffer variations}
	Let $C$ be a smooth complex projective curve of genus $g \geq 2$. Take a point $p \in C$ and fix a local coordinate $z$ centred in $p$. For $1 \leq n \leq 3g-3$, we define the $n^{th}$ Schiffer variation  at $p$ to be the element $\xi_p^n \in H^1(C, T_C) \cong H^{0,1}_{\bar{\partial}}(T_C)$ whose Dolbeault representative is given by $\frac{ \bar{\partial}b}{z^n} \frac{\partial}{\partial z}$, where $b$ is a bump function in $p$ which is equal to one in a small neighborhood $U$ containing $p$, $\xi_p^n = [\frac{ \bar{\partial}b}{z^n} \frac{\partial}{\partial z}]$. Clearly  $\xi_p^n$ depends on the choice of the local coordinate $z$. Take $1\leq n \leq 3g-3$. Consider the exact sequence 
	$$0 \ra T_C \ra T_C(np) \ra T_C(np)_{|np} \ra 0, $$
and the induced exact sequence in cohomology: 
$$ 0 \ra H^0(T_C(np)) \ra H^0(T_C(np)_{|np}) \stackrel{\delta^n_p}\lra H^1(T_C). $$
If either $n <2g-2$, or $n \leq 3g-3$ and $p$ is a general point, we have: $h^0(T_C(np)) =0$. In fact if  $n <2g-2$, $deg(T_C(np)) <0$, while if $p$ is a general point, $h^0(2K_C(-np)) = 3g-3 -n$, $\forall n \leq 3g-3$, hence by  Riemann Roch we have:  $h^0(T_C(np))  = h^0(2K_C(-np) ) -2g +2 +n -g +1 = 0. $  Hence we have an inclusion $\delta^n_p: H^0(T_C(np)_{|np})  \cong {\mathbb C}^n \hookrightarrow H^1(T_C)$ and the image of $\delta^n_p$ in $H^1(C,T_C)$ is the $n$-dimensional subspace generated by $\xi_p^1,...,\xi_p^n$.	
Hence a  basis for $H^1(C, T_C)$ is given by $\{\xi_p^1,..., \xi_p^{3g-3}\}$, where $p$ is a general point.

Now we give an immediate generalisation of Lemma 2.1 of \cite{cfg}. 

\begin{LEM}
Assume $1 \leq n \leq 3g-3$ and $p$ generic, or $1 \leq n <2g-2$ and $p \in C$ arbitrary, then the definition of $\xi_p^n$ depends on the choice of a local coordinate. In fact $\xi_p^1$ is only defined modulo a constant. In general the line generated by $\xi_p^n$ is defined modulo the subspace generated by $\{\xi_p^1,..., \xi_p^{n-1}\}$.

\end{LEM}
\begin{proof}
We have the following commutative diagram: 

 $$
\xymatrix@R=0.8cm@C=0.8cm{
& & &0 \ar[d]\\
             &  0 \ar[d]&  & T_C((n-1)p )_{|(n-1)p}\ar[d] & \\
    0 \ar[r]& T_C \ar[d] \ar[r] & T_C(np) \ar@{=}[d]  \ar[r]& T_C(np)_{|np}\ar[d]\ar[r]& 0\\
    0 \ar[r]& T_C((n-1)p) \ar[d] \ar[r] &T_C(np) \ar[r] & T_C(np)_{|p} \ar[d]\ar[r]& 0\\
     &T_C((n-1)p )_{|(n-1)p} \ar[d]  && 0& \\
    &0&&&
    }           
$$

Taking cohomology we get: 

 $$
\xymatrix@R=0.8cm@C=0.8cm{
& 0 \ar[d]& 0 \ar[d]&\\
&{\mathbb C}^{n-1}  \ar[d] \ar@{=}[r]&{\mathbb C}^{n-1}\ar[d]^{\delta^{n-1}_p}  &\\
         0 \ar[r]    &  {\mathbb C}^{n}  \ar[d] \ar[r]^{\delta^n_p}& H^1(T_C) \ar[r] \ar[d]^{\pi}& H^1(T_C(np ))\ar@{=}[d]  \ar[r] &0 \\
    0 \ar[r]& {\mathbb C}  \cong H^0(T_C(np)_{|p})  \ar[r]^{\overline{\delta^n_p}}\ar[d]& H^1(T_C((n-1)p)) \ar[d]  \ar[r]& H^1(T_C(np))\ar[r]& 0\\
    & 0  &0 && 
        }           
$$

So we have $Im({\overline{\delta^n_p)} = {\mathbb C} \cdot \pi (\xi_p^n)} \subset H^1(T_C((n-1)p)) \cong  H^1(T_C)/Im(\delta^{n-1}_p).$ Hence the line generated by $\xi_p^n$ is intrinsically  defined modulo  $Im(\delta^{n-1}_p) = \langle \xi_p^1,..., \xi_p^{n-1}\rangle$.

\end{proof}

From now on we will always consider higher Schiffer variations under the hypothesis: 

$$(*) \  \ \ 1 \leq n \leq 3g-3 \ and  \ p \in C \  generic, \ or \ 1 \leq n <2g-2 \ and  \ p \in C \ arbitrary.$$

\begin{LEM}
\label{xin}
Let $\beta \in H^0(2K_C)$. Let $p \in C$ and take as above $z$ a local coordinate centred in $p$. Assume that locally $\beta = f(z) dz^2$. Then if $p \in C$, we have: $\beta(\xi^k_p) = \langle \beta, \xi_p^k\rangle =  2 \pi i \frac{f^{(k-1)}(0)}{(k-1)!}$, for every $1 \leq k \leq 3g-3 $. 
\end{LEM}
\begin{proof}
Assume that $b\equiv 0$ outside the open subset $U$ containing $p$, then 
$$\langle \beta, \xi_p^k \rangle =	\int_C \beta \cup \frac{\bar{\partial}b}{z^k} = - \int_{U\setminus \{p\}} \bar{\partial}\big{(} \frac{b(z) f(z)}{z^k} dz\big{)}= - \lim_{\epsilon \ra 0} \int _{U \cap \{|z| >\epsilon \}} d \big{(} \frac{b(z) f(z)}{z^k} dz\big{)}=$$
$$=  \lim_{\epsilon \ra 0} \int _{|z| = \epsilon} \frac{f(z)}{z^k}dz = 2\pi i \frac{f^{(k-1)}(0)}{(k-1)!},$$
by the theorems of Stokes and Cauchy. 
\end{proof}

\begin{REM}
\label{bicanonical}
Assume hypothesis $(*)$ holds. Then the projective subspace ${\mathbb P}(\langle \xi^1_p, ..., \xi^n_p \rangle)$ has a geometric interpretation in terms of the bicanonical curve. Namely it is the $(n-1)^{th}$ osculating plane to the bicanonical curve in the point $p$. 
\end{REM}

\begin{proof}

First recall that the bicanonical map

$$\phi_{|2K_C|}: C \rightarrow {\mathbb P}(H^1(T_C))$$
can be seen as the map $p \mapsto {\mathbb P}(\langle \xi^1_p\rangle).$

In fact, $\phi_{|2K_C|}(p)=\{\sigma \in H^0(\omega_C^{\otimes 2}) \ | \ f(0)=0\}$, where in a coordinate $z$ centered in $p$ we have $\sigma = f(z) dz^2$.  
By   Lemma 2.1 of \cite{cfg} (and Lemma \ref{xin}), we have 
$$ \phi_{|2K_C|}(p)=\{\sigma \in H^0(\omega_C^{\otimes 2}) \ | \ f(0)=0\} = \{\sigma \in H^0(\omega_C^{\otimes 2}) \ |  \sigma(\xi_p^1) =  2 \pi i f(0)=0\}.$$

If we take an element $\lambda_1 \xi_p^1 + ...+ \lambda_n \xi_p^n \in {\mathbb P}(\langle \xi^1_p, ..., \xi^n_p \rangle)$, it corresponds to the hyperplane $\{\sigma \in H^0(\omega_C^{\otimes 2}) \ | \ \lambda_1 \sigma(\xi_p^1) + ...+ \lambda_n \sigma(\xi_p^n) =0\} = \{\sigma \in H^0(\omega_C^{\otimes 2}) \ | \lambda_1 f(0)+ \lambda_2 f'(0)+ ...+ \lambda_n  \frac{f^{(n-1)}(0)}{(n-1)!}  =0\}$, since by Lemma \ref{xin} we have $\sigma(\xi_p^k ) =\langle \sigma, \xi_p^k\rangle =  2 \pi i \frac{f^{(k-1)}(0)}{(k-1)!}$, $\forall k$. In coordinates, if we take a basis $\sigma_1, ..., \sigma_{3g-3}$ of $H^0(\omega_C^{\otimes 2})$, and we write locally $\sigma_i(z) = h_i(z) dz^2$, then we have 

$$ \{\sigma \in H^0(2K_C) \ | \lambda_1 f(0)+ \lambda_2 f'(0)+ ...+ \lambda_n  \frac{f^{(n-1)}(0)}{(n-1)!}  =0\} =$$ 
$$\{ (x_1,...,x_{3g-3})  \ | \ \lambda_1\sum_{i=1}^{3g-3}  x_i h_i(0) + ...+ \frac{\lambda_n}{(n-1)!} \sum_{i=1}^{3g-3}  x_i h^{(n-1)}_i(0)=$$
$$ =x_1 \sum_{i=1}^{n} \frac{\lambda_i}{(i-1)!} h_1^{(i-1)}(0) + ... + x_{3g-3} \sum_{i=1}^{n} \frac{\lambda_i}{(i-1)!} h_{3g-3}^{(i-1)}(0) =0\}.$$
So $\lambda_1 \xi_p^1 + ...+ \lambda_n \xi_p^n$ corresponds to the point in ${\mathbb P}^{3g-3}$: 
$$\lambda_1[h_1(0),...,h_{3g-3}(0)] + ...+ \frac{\lambda_n}{(n-1)!} [h_1^{(n-1)}(0),...,h_{3g-3}^{(n-1)}(0)],$$
which belongs to the $(n-1)^{th}$ osculating plane of the bicanonical image of $C$ at 
$\phi_{|2K_C|}(p) = [h_1(0),...,h_{3g-3}(0)]$. 

\end{proof}

\subsection{Second fundamental form and Hodge-Gaussian maps}
	
Let ${\mathcal M}_g$ denote the moduli space of smooth complex projective curves of genus $g \geq 4$ and  let ${\mathcal A}_g$ be the moduli space of principally polarised abelian varieties of dimension $g$. The moduli space ${\mathcal A}_g$ is a quotient of the Siegel space ${\mathcal H}_g = Sp(2g, {\mathbb R})/U(g)$ by the action of $Sp(2g, {\mathbb Z})$. The Siegel space ${\mathcal H}_g$ is a Hermitain symmetric domain of the non-compact type and thus it has a canonical symmetric metric. The quotient ${\mathcal A}_g$ is hence endowed with the induced  orbifold metric that we call the Siegel metric. 

Assume $g \geq 4$. The Torelli map
$$j: {\mathcal M}_g \to {\mathcal A}_g,  \ [C] \mapsto [j(C), \Theta_C],$$
where $j(C)$ is the Jacobian of $C$ and $\Theta_C$ is the principal polarisation induced by cup product, is an orbifold embedding  outside the hyperelliptic locus (\cite{oort-st}). 

Denote by ${\mathcal M}_g^0$ the complement of the hyperelliptic locus in ${\mathcal M}_g$ and consider  the cotangent exact sequence of the Torelli map: 
$$ 0 \to N_{{\mathcal M}_g^0/{\mathcal A}_g} ^*\to \Omega^1_{{\mathcal A}_g{|{\mathcal M}_g^0}} \stackrel{q}\to \Omega^1_{{\mathcal M}_g^0} \to 0,$$
where $q = dj^*$ is the dual of the differential of the Torelli map. 
Let $\nabla$ be  the Chern connection on $\Omega^1_{{\mathcal A}_g{|{\mathcal M}_g^0}}$ with respect to the Siegel metric and  let 
$$II^*:  N_{{\mathcal M}_g^0/{\mathcal A}_g} ^* \to Sym^2 \Omega^1_{{\mathcal M}_g^0}, \ \ II^* = (q \otimes Id_{\Omega^1_{{\mathcal M}_g^0}}) \circ \nabla_{|N_{{\mathcal M}_g^0/{\mathcal A}_g} ^*}$$
be the second fundamental form of the above exact sequence. 
  
 At a point $[C] \in {\mathcal M}_g^0$, we have the following identifications: 
 $$ N_{{\mathcal M}_g^0/{\mathcal A}_g, [C]} ^*= I_2(C, \omega_C), \ \Omega^1_{{\mathcal A}_g{|{\mathcal M}_g^0}, [C]} = Sym^2H^0(C, \omega_C), \ \Omega^1_{{\mathcal M}_g^0, [C]} = H^0(C, \omega_C^{\otimes 2}),$$
 where $I_2(C, \omega_C)$ is the vector space of quadrics containing the canonical curve and the codifferential of the Torelli map $q$ is the multiplication map of global sections. Then, at the point $[C]$, the second fundamental form is a linear map
 $$II^*: I_2(C, \omega_C) \to Sym^2H^0(C, \omega_C^{\otimes 2}). $$
 
 In  \cite[Theorem 2.1]{cpt} it is proven that $II^*$ is equal (up to a constant) to the Hodge-Gaussian map $\rho$ of  \cite[Proposition-Definition 1.3]{cpt}.

 Let us briefly recall the definition of $\rho$. 
Take a point $p \in C$ and a local coordinate $z$ centered in $p$.  Set, as above,  $\theta_n := \frac{ \bar{\partial}b}{z^n} \frac{\partial}{\partial z}$, so that $\xi_p^n = [\theta_n]$, and take a basis $\omega_1, ..., \omega_g$ of $H^0(\omega_C)$. Write locally around $p$, $\omega_i = f_i(z) dz$. Then the contraction $\theta_n \omega_j $ is a $(0,1)$-form on $C$, so we write 
$$\theta_n \omega_j = \gamma^n_j  + \bar{\partial} h^n_j,$$
where $\gamma^n_j \in \Lambda^{0,1}(C)$ is harmonic. 
Take  a quadric $Q \in I_2(K_C)$, then $Q = \sum_{i,j=1}^g a_{ij} (\omega_i \otimes \omega_j)$, where $a_{ij}= a_{ji}$, $\forall i,j =1,...,g$ and $\sum_{i,j} a_{ij} f_i(z) f_j(z) \equiv 0$. Then by \cite[Proposition-Definition 1.3]{cpt} we have: 
\begin{equation}
\rho(Q) (\xi_p^n) =\big{[} \sum_{ij} a_{ij} \omega_i \partial h^n_j\big{]} \in H^0(C, \omega_C^{\otimes 2}).
\end{equation}
Notice that locally we have 
$\theta_n\omega_j = \frac{\bar{\partial}b}{z^n} f_j(z) = \bar{\partial}\big{(}\frac{b f_j}{z^n}\big{)}.$
This expression is in fact global. In fact, since $\theta_n \omega_j = \gamma_j^n  + \bar{\partial} h^n_j,$ we get 
$\gamma_j^n = \bar{\partial} \big{(} \frac{b f_j}{z^n} - h^n_j\big{)}.$
Set $$g_j^n := \frac{bf_j}{z^n} - h^n_j,$$
then $\gamma^n_j = \bar{\partial}g^n_j$. Set $\eta^n_j:= \partial g^n_j$. Notice that $\eta^n_j \in H^0(C, \omega_C((n+1)p))$. In fact, it has a pole of order $n+1$ in $p$ and it is holomorphic elsewhere, since $\overline{\partial} 
(\eta^n_j) = \overline{\partial} (\partial g^n_j)  = - \partial  \overline{\partial}(g^n_j) = - \partial (\gamma^n_j ) = 0$, since $\gamma^n_j $ is harmonic. 
Now we recall the following result of \cite{cpt}: 

\begin{PROP}
We have: $\sum_{ij} a_{ij} \omega_i \partial h^n_j= -\sum_{ij} a_{ij} \omega_i \eta^n_j$, hence 
$$\rho(Q) (\xi_p^n)= \big{[} -\sum_{ij} a_{ij} \omega_i \eta^n_j\big{]}.$$
\end{PROP}
\proof  
We have $h_j^n := \frac{bf_j}{z^n} - g^n_j,$ so $\partial h_j^n =\partial \big{(}\frac{b f_j}{z^n}\big{)} - \partial g^n_j = \partial \big{(}\frac{b f_j}{z^n}\big{)} - \eta^n_j $. Thus it is enough to show that $\sum_{ij} a_{ij} \omega_i \partial \big{(}\frac{b f_j}{z^n}\big{)} =0$. To this purpose, notice that $\sum_{ij} a_{ij} \omega_i \partial \big{(}\frac{b f_j}{z^n}\big{)} $ is a holomorphic section of $\omega_C^{\otimes 2}$ on $C\setminus \{p\}$, since 
$\sum_{ij} a_{ij} \omega_i \partial \big{(}\frac{b f_j}{z^n}\big{)} = \sum_{ij} a_{ij} \omega_i \partial h^n_j + \sum_{ij} a_{ij} \omega_i \eta^n_j$ and $ \sum_{ij} a_{ij} \omega_i \partial h^n_j$ is $\bar{\partial}$-closed  by \cite[Proposition-Definition 1.3]{cpt} and $\bar{\partial} \eta^n_j = \bar{\partial} \partial g^n_j = - \partial \bar{\partial} g^n_j = -\partial \gamma^n_j =0$. So $\sum_{ij} a_{ij} \omega_i \partial \big{(}\frac{b f_j}{z^n}\big{)}$ is holomorphic on $C \setminus \{p\}$ and if we set $V:= C\setminus supp(b)$, then $V$ is a non-empty open subset of $C \setminus \{p\}$, where $b$ is identically zero. Therefore $\sum_{ij} a_{ij} \omega_i \partial \big{(}\frac{b f_j}{z^n}\big{)}\equiv 0$ on $V$, so it is zero everywhere. 
\qed

\subsection{The forms $\eta^n_j$}

Let $p \in C$, then we have the following isomorphism $H^1(C\setminus \{p\}, \mathbb C) \cong H^1(C, \mathbb C)$ and the Hodge decomposition $H^1(C, \mathbb C) \cong H^{1,0}(C) \oplus H^{0,1}(C)$.  So, denote by 
$$\phi_n: H^0(C,\omega_C((n+1)p)) \ra H^1(C\setminus \{p\}, \mathbb C) \cong H^1(C, \mathbb C) \cong H^{1,0}(C) \oplus H^{0,1}(C)$$ the map which associates to a meromorphic form $\alpha \in H^0(C, \omega_C(n+1)p)$ its cohomology class $[\alpha] \in H^1(C, \mathbb C)$. 

We have already observed that $\eta^n_j= \partial g^n_j =\partial \big{(}\frac{b f_j}{z^n} - h^n_j\big{)} \in H^0(C,\omega_C((n+1)p))$. 
Since $\eta^n_j + \gamma^n_j = \bar{\partial}g^n_j+\partial g^n_j = dg^n_j$, then $[\eta^n_j ] = - [\gamma^n_j] \in H^{0,1}(C)$. 	Therefore the form  $\eta^n_j$ belongs to $\phi_n^{-1}(H^{0,1}(C))$.

\begin {LEM} The subspace  $\phi_n^{-1}(H^{0,1}(C)) \subset H^0(C,\omega_C((n+1)p))$ has dimension $n$. 
If $n \leq g$, then $\phi_n$ is injective if and only if $h^0(np) =1$. So, if $p$ is not a Weierstrass point, the map $\phi_n$ is injective if $n \leq g$, and surjective if $n \geq g$. 
\end{LEM}

\proof
Consider the composition of the map $\phi_n$ with the projection $\pi: H^1(C, {\mathbb C} ) \rightarrow H^{1,0}(C)$: $$\pi \circ \phi_n: H^0(C, \omega_C((n+1)p)) \rightarrow H^{1,0}(C).$$ The restriction $\pi \circ {\phi_n}_{|H^0(C,\omega_C)}:  H^0(C, \omega_C) \rightarrow  H^{1,0}(C)$ is an isomorphism, hence $ker (\pi \circ \phi_n )= \phi_n^{-1}(H^{0,1}(C))$, which  has dimension equal to $h^0(\omega_C((n+1)p)) - g = g+n-g = n$. 

Assume  $n \leq g$. If $\alpha \in H^0(C,\omega_C((n+1)p))$ and $\alpha = df$, then $f$ would be a meromorphic function on $C$ with a pole of order at most $n$ in $p$ and holomorphic elsewhere, so $f \in H^0(np)$, and therefore $\alpha = 0$ if and only if $h^0(np) =1$. Assume now that $p$ is not a Weierstrass point. Then $h^0(np) = 1$, $\forall n \leq g$, so $\phi_n$ is injective $\forall n \leq g$, and if $n=g$ it  is an isomorphism, since it is injective and $h^0(C,\omega_C((g+1)p)) = 2g = \dim H^1(C, \mathbb C)$. For $n \geq g$ the map $\phi_n$ is surjective, since $H^0(C,\omega_C((n+1)p))$ contains $H^0(C,\omega_C((g+1)p))$. 
\qed\\

\begin{REM}
\label{eta_pk}
Notice that, since the restriction of $\phi_n$  to $H^0(C, \omega_C(np))$ coincides with $\phi_{n-1}$, we have $$\phi_{n-1}^{-1}(H^{0,1}(C)) \subset \phi_n^{-1}(H^{0,1}(C)),$$

and the codimension is one. 

Now fix a local coordinate $z$ centred in $p$. We can construct a basis for $\phi_n^{-1}(H^{0,1}(C)) \subset H^0(C,\omega_C((n+1)p))$  given by  forms $\{\eta_{z,k} \}_{k =1,...,n}$, whose local expression around $p$  is $\eta_{z,k} = (\frac{1}{z^{k+1}} + r_k(z))dz$, where $r_k(z)$ is holomorphic. 

In fact $\eta_{z,1} = \eta_z$ as in \cite{cfg}, whose local expression is  $(\frac{1}{z^{2}} + r_1(z))dz$ . Since $\phi_1^{-1}(H^{0,1}(C))$ has codimension 1 in $\phi_2^{-1}(H^{0,1}(C))$, there exists a form $\beta_2 \in \phi_2^{-1}(H^{0,1}(C))$ whose local expression in $p$ is $(\frac{1}{z^3} + \frac{b}{z^2} + g(z)) dz$ with $g$ holomorphic. Then we can take $\beta_2 - b \eta_{z,1} = \eta_{z,2}$.  Going on in this way we find a basis of $\phi_n^{-1}(H^{0,1}(C)) \subset H^0(C,\omega_C((n+1)p))$ as stated.  

\end{REM}
Notice that this basis depends on the choice of the local coordinate $z$. 

Indeed, take anothere local coordinate $w$ centred in $p$, and consider the corresponding form $\eta_{w,n} = (\frac{1}{w^{n+1}} + s_n(w)) dw$. If we change coordinates,  $w = h(z) = \lambda z + O(z^2)$, we have: 
$$\eta_{w,n} = \big{(}\frac{1}{w^{n+1}} + s_n(w)\big{)} dw= \big{(}\frac{1}{(h(z))^{n+1}} + s_n(h(z))\big{)} h'(z) dz.$$ We have $\frac{h'(z)}{(h(z))^{n+1}}  =\frac{1}{\lambda^n}(\frac{1}{z^{n+1}} + \psi(z))$, where $\psi(z)$ is a meromorphic function with poles of order $\leq n$ in $0$. 

Then we have the following equality modulo $\phi_{n-1}^{-1}(H^{0,1}(C)) \subset H^0(\omega_C(np)):$

$$\eta_{w,n} = \frac{1}{\lambda^n} \eta_{z,n} = \left(\frac{dz}{dw}\right)_0^n \cdot \eta_{z,n} \ \ \mod \phi_{n-1}^{-1}(H^{0,1}(C)) \subset H^0(\omega_C(np)).$$ 
Therefore, if $n=1$, since $H^0(C, \omega_C(p)) =  H^0(C, \omega_C)$, $ \phi_{0}^{-1}(H^{0,1}(C)) = \{0\}$, we have 

$$\eta_{w,1} = \frac{1}{\lambda} \eta_{z,1} = \big{(}\frac{dz}{dw}\big{)}_0 \cdot \eta_{z,1}.$$ 
So one can define a linear map 
$${\eta_{p,1}} : T_{C,p} \to H^0(C,\omega_C(2p)), \ \ \ {\eta_{p,1}}(\lambda (\frac{\partial}{\partial z})(p)) = \lambda { \eta_{z,1}}, $$ (see \cite[2.4]{cfg}).

If $n \geq 2$, we only get a linear map 
$$\overline{\eta_{p,n}} : T^{\otimes n}_{C,p} \to H^0(C,\omega_C((n+1)p))/H^0(C,\omega_C(np)),$$
$$\overline{\eta_{p,n}}(\lambda (\frac{\partial}{\partial z})^n(p)) = \lambda { \eta_{z,n}}= \lambda(\frac{1}{z^{n+1}} + r_n(z))dz  \ \  \mod \phi_{n-1}^{-1}(H^{0,1}(C)) \subset H^0(\omega_C(np)).$$

On the surface $S = C\times C$ consider the two projections $p, q: C \times C \to C$, and the holomorphic vector bundle $E = \omega_C \otimes p_*(q^*\omega_C(2\Delta))$ on the curve $C$. 
If $n =1$, the map $p \mapsto \eta_{p,1}$ gives  a section $\eta$ of the vector bundle  $\omega_C \otimes p_*(q^*\omega_C(2\Delta))$. In \cite[Prop. 3.4]{cfg} it is proven that $\eta$ is a holomorphic section of $\omega_C \otimes p_*(q^*\omega_C(2\Delta))$.

Let us now describe the forms $\eta^n_j= \partial g^n_j =\partial \big{(}\frac{b f_j}{z^n} - h^n_j\big{)} $ locally around $p$. Since $b \equiv 1$ in a neighborhood of $p$, we have 
$$\eta^n_j =\partial \big{(}\frac{ f_j}{z^n}\big{)}  + \psi^n_j(z) dz= \frac{z f'_j(z) - n f_j(z)}{z^{n+1}}dz +  \psi^n_j(z) dz, $$
where $\psi^n_j(z)$ is a ${\mathcal C}^{\infty}$ function. 

We have 
$$\frac{z f'_j(z) - n f_j(z)}{z^{n+1}} +  \psi^n_j(z) = \frac{1}{z^{n+1}} \big{(} \sum^{n-1}_{k=0} f_j^{(k+1)}(0) \frac{z^{k+1}}{k!} - n \sum^n_{k =0} f_j^{(k)}(0) \frac{z^k}{k!} \big{)}+  l^n_j(z)=  $$
$$=  \frac{1}{z^{n+1}} \big{(} \sum^n_{k = 1} f_j^{(k)}(0) \frac{z^{k}}{(k-1)!} - n \sum^n_{k = 0} f_j^{(k)}(0) \frac{z^k}{k!} \big{)} + l^n_j(z)=$$
$$=  \frac{1}{z^{n+1}} \big{(} \sum^n_{k =1} f_j^{(k)}(0) \frac{z^{k}}{k!}(n-k) - n f_j(0)\big{)} +  l^n_j(z)=$$
$$= \frac{1}{z^{n+1}} \big{(} \sum^{n-1}_{k = 0} f_j^{(k)}(0) \frac{k-n}{k!}z^{k} \big{)}+   l^n_j(z)= \sum^{n-1}_{k = 0} f_j^{(k)}(0) \frac{k-n}{k!}\frac{1}{z^{n+1-k}} +   l^n_j(z),$$
where $ l^n_j(z)$ is a ${\mathcal C}^{\infty}$ function.

From now on we set $b_{jk} := \frac{f_j^{(k)}(0)}{k!}$, so that $f_j (z) = \sum_{k=0}^{\infty} b_{jk}z^k$. 
\begin{PROP}
\label{etan}
We have
\begin{enumerate}
\item $$\eta^n_j = \sum^{n-1}_{k =0} (k-n)b_{jk}\cdot \eta_{z, n-k}.$$
\item The  function $ l^n_j(z) = \sum^{n-1}_{k=0} (k-n)b_{jk} \cdot  r_{n-k}(z), $ so
$$\eta^n_j =\sum^{n-1}_{k = 0} (k-n) b_{jk}\frac{dz}{z^{n+1-k}} + \sum^{n-1}_{k=0}  (k-n) b_{jk} \cdot  r_{n-k}(z)dz.$$

\end{enumerate}
\end{PROP}
\proof 
The forms $\eta^n_j $ belong to the $n$-dimensional subspace $$\phi_n^{-1}(H^{0,1}(C)) \subset H^0(C,\omega_C((n+1)p)).$$  

So, using the basis given in Remark \ref{eta_pk} and the above computation, we have 
$$\eta^n_j =\sum^{n-1}_{k =0}(k-n) b_{jk}\frac{dz}{z^{n+1-k}} +   l^n_j(z) dz = a_1 \eta_{z,1} + ...+ a_{n} \eta_{z, n}= $$
$$=(\frac{a_1}{z^2} + ...+ \frac{a_{n}}{z^{n+1}} + a_1 r_1(z) +...+ a_{n}r_{n}(z)) dz.$$
Therefore $a_{n-k} = (k-n) b_{jk}$, $\forall k=0,...,n-1$, and 
$$l^n_j(z) = \sum^{n-1}_{k=0} (k-n) b_{jk} \cdot r_{n-k}(z),$$ 
hence $(1)$ and $(2)$ immediately follow. 

\qed

\begin{REM}
Notice that equality $(1)$ of Proposition \ref{etan} implies that 
$$\eta_j = \eta^1_j = -f_j(0)\eta_{p,1},$$ where $\eta_{p,1} = \eta_p$ (see \cite{cfg}). 
\end{REM}

\subsection{Gaussian maps}

We briefly recall the definition of Gaussian maps for curves (see \cite{wahl90}, \cite{wahl92}). Let $N$
and $M$ be two line bundles on $C$.  Set $S:=C\times C$ and $\Delta
\subset S$ the diagonal. For a non-negative integer $k$ the
\emph{k-th Gaussian} or \emph{Wahl map}  is
the map given by restriction to the diagonal
\begin{equation*}
  H^0(S,    N\boxtimes M (-k \Delta) ) \stackrel{    \mu^k_{N,M}}
  {\longrightarrow}
  H^0(S, N\boxtimes M (-k \Delta)_{|{\Delta})}\cong
  H^0(C, N\otimes M \otimes \omega_C^{{\otimes k}}).
\end{equation*}
We will be only interested in the case $N=M$.  In this case we set
$\mu_{k,M}:=\mu^k_{M,M}$.  We have: $H^0(S, N
\boxtimes M) \cong H^0(C,N) \otimes H^0(C,M)$, so the map $\mu_{0,M}$ is
the multiplication map of global sections, 
\begin{equation*}
  H^0(C,M)\otimes
  H^0(C,M)\rightarrow H^0(C,M^{\otimes  2}),
\end{equation*}
which vanishes identically on $\wedge^2 H^0(C,M)$.
Thereofore $\ker \mu_{0,M}$ $= H^0(S, M\boxtimes M(-\Delta))$
decomposes as $\wedge^2 H^0(C,M)\oplus I_2(M)$, where $I_2(M)$ is the
kernel of $Sym^2H^0(C,M)\rightarrow H^0(C,M^{\otimes 2})$. Since $\mu_{1,M}$
vanishes on symmetric tensors, we write
\begin{equation}\nonumber
  \mu_{1,M}:\wedge^2H^0(C,M)\rightarrow H^0(C,
  \omega_C\otimes M^{ \otimes 2}).
\end{equation}
If $\sigma$ is a local frame for $M$ around $p$, $z$ is a local coordinate centred in $p$,  and we take two sections $s_1, s_2 \in H^0(C,M)$, whose local expressions are $s_i = f_i(z) \sigma$, we
have
\begin{gather}
  \label{muformula}
  \mu_{1,M} (s_1\wedge s_2) = (f'_1 f_2 - f'_2f_1) dz \otimes \sigma^{
    2} .
\end{gather}

The vector space  $H^0(S, M\boxtimes M (-2\Delta))$ decomposes as the sum of
$I_2(M)$ and the kernel of $\mu_{1,M}$. Since $\mu_{2,M}$ vanishes
identically on skew-symmetric tensors, we write
\begin{equation*}
  \mu_{2,M}: I_2(M)\rightarrow H^0(C,M^{\otimes 2}\otimes \omega_C^{\otimes 2}).
\end{equation*}
We will denote by \begin{equation*}
  \mu_2:= \mu_{2,\omega_C}:I_2(\omega_C)\rightarrow H^0(C, \omega_C^{\otimes 4}).
\end{equation*}

the second Gaussian map of the canonical bundle on $C$ and in general 
$\mu_k := \mu_{k,\omega_C}$. 

Choose a  local coordinate $z$, take a basis $\{\omega_1, ..., \omega_g\}$ of $H^0(C, \omega_C)$ whose local expression is  $\omega_i = f_i(z) dz$. Take a quadric $Q = \sum_{i,j=1}^g a_{ij} \omega_i \otimes \omega_j   \in I_2(\omega_C)$, where $a_{ij}= a_{ji}$, $\forall i,j =1,...,g$ and $\sum_{i,j} a_{ij} f_i(z) f_j(z) \equiv 0$. 

Then we have: 
$\sum_{i,j=1}^g a_{ij} f_i(z) f'_j(z) \equiv 0, $ and 
$$\mu_2(Q) = \sum_{i,j=1}^g a_{ij} f_i(z) f''_j(z)dz^4 = -\sum_{i,j=1}^g a_{ij} f'_i(z) f'_j(z) dz^4.$$
In \cite[Theorem 3.1]{cpt} it is proven that, up to a constant, we have $\rho \circ m = \mu_2$, where $m: Sym^2H^0(C, \omega_C^{\otimes 2}) \to H^0(C, \omega_C^{\otimes 4})$ is the multiplication map. 

Assume now that $Q \in ker (\mu_2)$. Then we have: 
$$\sum_{ij} a_{ij} f_i(z) f_j(z) = 0, \ \sum_{ij} a_{ij} f'_i(z) f_j(z) = 0, \ \sum_{ij} a_{ij} f^{(2)}_i(z) f_j(z) = - \sum_{ij} a_{ij} f'_i(z) f'_j(z) =0,$$ 
$$\sum_{ij} a_{ij} f^{(3)}_i(z) f_j(z) = - \sum_{ij} a_{ij} f^{(2)}_i(z) f'_j(z) = 0,$$
and 
$$\sum_{ij} a_{ij} f^{(4)}_i(z) f_j(z) = - \sum_{ij} a_{ij} f^{(3)}_i(z) f'_j(z) = \sum_{ij} a_{ij} f^{(2)}_i(z) f^{(2)}_j(z) = \mu_4(Q).$$

In general, the condition $Q \in ker(\mu_m)$ with $m$ even, is equivalent to 
\begin{equation}
\label{mu}
\sum_{i,j=1}^g a_{ij} f_i^{(h)}(z) f_j^{(k)} (z) \equiv 0, \ \ \forall h,k, \ \text{such that } \ h+k \leq m+1.
\end{equation}

Moreover, $\forall k =0, ..., m+2$,  we have: 
\begin{equation}
\label{rimu}
\mu_{m+2}(Q) = (-1)^k \sum_{i,j=1}^g a_{ij} f_i^{(m+2-k)}(z) f_j^{(k)} (z) (dz)^{m+4}.
\end{equation}

\section{Main Theorem}

Consider a smooth complex projective curve $C$ of genus at least 4. 
In the following computations we will always assume hypothesis (*).

Let us now consider the element $ \beta_n := \rho(Q)(\xi^n_p) = - \sum_{i,j} a_{ij} \omega_i \eta^n_j  \in H^0(C, \omega_C^{\otimes 2})$, where $Q = \sum_{ij} a_{ij} \omega_i \otimes \omega_j  \in I_2(K_C)$. 
In a local coordinate $z$ centred in $p$, we can write  $\beta_n = \rho(Q)(\xi^n_p) = \Phi_{n,Q}(z)dz^2$. We will now compute the function $\Phi_{n,Q}(z)$. 
By Proposition \ref{etan} we have: 
$$
\beta_n = \rho(Q)(\xi^n_p) = - \sum^g_{i,j=1} a_{ij} \omega_i \eta^n_j =$$
$$ - (\sum^g_{i,j=1} a_{ij} f_i(z)dz)(k-n)\sum^{n-1}_{k = 0}   b_{jk}(\frac{1}{z^{n+1-k} }+ r_{n-k}(z))dz=$$
$$=\sum^g_{i,j=1} a_{ij}(\sum_{h \geq 0}b_{ih}z^h)(\sum^{n-1}_{k = 0}(n-k) b_{jk}(\frac{1}{z^{n+1-k}} + r_{n-k}(z)))dz^2.$$
Hence 
$$\Phi_{n,Q}(z) = \sum^g_{i,j=1} \sum_{h \geq 0}\sum^{n-1}_{k = 0}(n-k)a_{ij}b_{jk}b_{ih}(z^{h+k-n-1} + r_{n-k}(z) z^h).$$

Before looking at the general case, let us make some explicit computations in  the case $n=1$. 
Recall that $Q \in I_2(K_C) = ker(\mu_0)$, hence $\sum_{ij} a_{ij} f_i(z) f_j(z) \equiv 0$ and therefore $\sum_{ij} a_{ij} f'_i(z) f_j(z) \equiv 0.$  Hence $\sum_{ij} a_{ij} b_{ih}b_{j0} = 0$, for $h =0,1$.

So, for $n=1$ we have: 

$$\Phi_{1,Q}(z) =  \sum^g_{i,j=1} \sum_{h \geq 0}a_{ij} b_{ih}b_{j0}(z^{h-2} + r_1(z) z^h) 
=\sum^g_{i,j=1} \sum_{h \geq 2}a_{ij} b_{ih}b_{j0}(z^{h-2} + r_1(z) z^h).$$
Thus, by Lemma \ref{xin} we have 

$$\rho(Q) (\xi_p)(\xi_p) = \beta_1(\xi_p) = 2 \pi i \Phi_{1,Q}(0)=   \pi i \sum^g_{i,j=1} a_{ij} f^{(2)}_i(0) f_j(0) =  \pi i \mu_2(Q)(p),$$
as it is proven in \cite[Theorem 3.1]{cpt}. 

Assume now that $Q \in ker (\mu_2)$. Then we have: 
$$\sum_{ij} a_{ij} f_i(z) f_j(z) = 0, \ \sum_{ij} a_{ij} f'_i(z) f_j(z) = 0, \ \sum_{ij} a_{ij} f^{(2)}_i(z) f_j(z) = - \sum_{ij} a_{ij} f'_i(z) f'_j(z) =0,$$ 
$$\sum_{ij} a_{ij} f^{(3)}_i(z) f_j(z) = - \sum_{ij} a_{ij} f^{(2)}_i(z) f'_j(z) = 0,$$
hence $\sum_{ij} a_{ij} b_{ih} b_{j0} = 0$, for $h \leq 3$. 
Moreover we have: 
$$\sum_{ij} a_{ij} f^{(4)}_i(z) f_j(z) = - \sum_{ij} a_{ij} f^{(3)}_i(z) f'_j(z) = \sum_{ij} a_{ij} f^{(2)}_i(z) f^{(2)}_j(z) = \mu_4(Q).$$ 

Then, by Lemma \ref{xin}, $\beta_1(\xi^2_p) = 2 \pi i \Phi'_{1,Q}(0).$
Since $Q \in Ker(\mu_2)$, we have:
    $$\Phi_{1,Q}(z) = \sum^g_{i,j=1} \sum_{h \geq 4}a_{ij} b_{ih}b_{j0}(z^{h-2} + r_1(z) z^h),$$
So 

$$\rho(Q)(\xi_p)(\xi^2_p) = \beta_1(\xi^2_p) = 2 \pi i \Phi'_{1,Q}(0)= 0, $$
$$\rho(Q)(\xi_p) (\xi^3_p) =\beta_1(\xi^3_p) =  \pi i \Phi^{(2)}_{1,Q}(0)=  2\pi i \sum^g_{i,j=1} a_{ij} f^{(4)}_i(0) f_j(0)\frac{1}{4!} =  \frac{2\pi i}{4!} \mu_4(Q)(p). $$




In general, we have the following  

\begin{PROP}
\label{mainprop}
If $Q \in ker \mu_m$, $m$ even, $m \geq 2$, then 
$$ \rho(Q)(\xi^n_p)(\xi^l_p) = 0, \ if \ l+n  \leq m+1, \ \forall l,n \geq 1,$$
$$ \rho(Q)(\xi^n_p)(\xi^{m+2-n)}_p) = c_{n,m}  \cdot \mu_{m+2}(Q)(p), \ \forall n<m+2,$$
where $c_{1,m} =\frac{2 \pi i}{(m+2)!}$, while for $n \geq 2$, we have 
 $c_{n,m} =  2 \pi i \cdot (-1)^{n-1} \cdot \frac{\prod_{j=0}^{n-2} (m-j)}{(n-1)!(m+2)!}$. 

\end{PROP}

\begin{proof}

Recall that by $\eqref{mu}$, $\eqref{rimu}$, the condition $Q \in ker(\mu_m)$ with $m$ even, is equivalent to 
$$\sum_{i,j=1}^g a_{ij} f_i^{(h)}(z) f_j^{(k)} (z) \equiv 0, \ \ \forall h,k, \ \text{such that } \ h+k \leq m+1,
$$
hence 
\begin{equation}
\label{notation}
\sum_{i,j=1}^g a_{ij} b_{ih} b_{jk} \equiv 0, \ \ \forall h,k, \ \text{such that } \ h+k \leq m+1,
\end{equation}

and  $\forall k =0, ..., m+2$,  we have: 
$$
\mu_{m+2}(Q) = (-1)^k \sum_{i,j=1}^g a_{ij} f_i^{(m+2-k)}(z) f_j^{(k)} (z) (dz)^{m+4}.
$$
We have computed 
$$\Phi_{n,Q}(z) = \sum^g_{i,j=1} \sum_{h \geq 0}\sum^{n-1}_{k = 0}(n-k)a_{ij}b_{jk}b_{ih}(z^{h+k-n-1} + r_{n-k}(z) z^h).$$

By Lemma \ref{xin} and using \eqref{mu}, for $n+l \leq m+2$, we have $l-1 \leq m+1-n \leq m$, so 
$$ \rho(Q)(\xi^n_p)(\xi^l_p) = \frac{2 \pi i}{(l-1)! } \Phi^{(l-1)}_{n,Q}(0) = $$
$$  2 \pi i \cdot \left( \sum_{h\geq 0, k\leq n-1,h+k = n+l} \left (  \sum^g_{i,j=1}(n-k) a_{ij} b_{jk} b_{ih}\right)+\sum^{n-1}_{k = 0}  \left(\sum^g_{i,j=1}  a_{ij} b_{jk} b_{i l-1}\right)r_{n-k}(0)\right).$$


So if $n+l \leq m+1$, both sums are zero thanks to  \eqref{notation}.  
If $n+l = m+2$, only the first sum is non zero and we obtain:
$$ \rho(Q)(\xi^n_p)(\xi^l_p) = \frac{2 \pi i}{(l-1)! } \Phi^{(l-1)}_{n,Q}(0) = 2\pi i \cdot  \sum_{k=0}^{n-1}   \sum^g_{i,j=1} (n-k)a_{ij}b_{jk} b_{i, m+2-k}=$$
$$= 2 \pi i  \left( \sum_{k=0}^{n-1} \left (  \sum^g_{i,j=1}a_{ij} f^{(m+2-k)}_i(0) f_j^{(k)}(0)\right)\frac{(k-n)}{k!(m+2-k)!}\right)=$$
$$=  2 \pi i  \cdot  \mu_{m+2}(Q)(p) \cdot  \sum_{k = 0}^{n-1} (-1)^k\frac{(n-k)}{k!(m+2-k)!}= c_{n,m} \mu_{m+2}(Q)(p),$$
by \eqref{rimu}, where $c_{n,m}=  2 \pi i \sum_{k = 0}^{n-1} (-1)^k\frac{(n-k)}{k!(m+2-k)!}$. 
Then clearly $c_{1,m} = \frac{2 \pi i}{(m+2)!}$, while for $n \geq 2$, we have 
$$c_{n,m} = 2 \pi i \cdot (-1)^{n-1}  \cdot \frac{\prod_{j=0}^{n-2} (m-j)}{(n-1)!(m+2)!}.$$

To prove this equality, first one easily proves by induction on $t \geq 1$, that $\forall l \geq t+1$, we have 
\begin{equation}
\label{teschio}
\sum_{k = 0}^{t} \frac{(-1)^k}{k!(l-k)!} = \frac{(-1)^t \prod_{k=1}^t (l-k)}{l!t!}.
\end{equation}
Then we have  
$$\sum_{k = 0}^{n-1} (-1)^k\frac{(n-k)}{k!(m+2-k)!} = n \sum_{k = 0}^{n-1} \frac{(-1)^k}{k!(m+2-k)!} - \sum_{k = 0}^{n-1} \frac{k (-1)^k}{k!(m+2-k)!}= $$
$$= n \sum_{k = 0}^{n-1} \frac{(-1)^k}{k!(m+2-k)!} + \sum_{h = 0}^{n-2} \frac{(-1)^h}{h!(m+1-h)!}.$$
So, applying \eqref{teschio} on both terms above, we get: 
$$\sum_{k = 0}^{n-1} (-1)^k\frac{(n-k)}{k!(m+2-k)!}= n  \frac{(-1)^{n-1} \prod_{k=1}^{n-1} (m+2-k)}{(n-1)!(m+2)!} +  \frac{(-1)^{n-2} \prod_{k=1}^{n-2} (m+1-k)}{(n-2)!(m+1)!}=$$
$$= \frac{(-1)^{n-1}}{(n-1)!(m+2)!} \left( n \prod_{h=0}^{n-2}(m+1-h) - (n-1)(m+2) \prod_{k=1}^{n-2}(m+1-k) \right)=$$
$$= \frac{(-1)^{n-1}}{(n-1)!(m+2)!}\left(\prod_{h=1}^{n-2}(m+1-h ))\right) (m+1-(n-1))= \frac{(-1)^{n-1}}{(n-1)!(m+2)!} \cdot \prod_{j=0}^{n-2}(m-j).$$
\end{proof}

\begin{REM}
With the same computation as in the proof of Proposition \ref{mainprop} one can immediately prove that for any quadric $Q \in I_2(\omega_C)$ such that 
$$\sum_{i,j=1^g} a_{ij} f_i^{(h)}(0) f_j^{(k)}(0)= 0, \ \forall h,k \geq 0, \ h +k \leq m,$$ 
then 
$$\rho(Q)(\xi_p^n)(\xi_p^l) = 0, \ if \  l+n  \leq m, \ \forall l,n \geq 1,$$
and if $l +n = m+1$, then 
$$ \rho(Q)(\xi^n_p)(\xi^l_p) =  2 \pi i  \left( \sum_{k=0}^{n-1} \left (  \sum^g_{i,j=1}a_{ij} f^{(m+1-k)}_i(0) f_j^{(k)}(0)\right)\frac{(n-k)}{k!(m+1-k)!}\right).$$

\end{REM}
So we can prove the following

\begin{THEOR}
\label{thmA}
Take $p$ a general point in $C$ and $1\leq n \leq 3g-3$. Then for every $Q \in ker(\mu_{2n-2})$, and for every $v, w \in \langle \xi^1_p,..., \xi^{n-1}_p \rangle$, we have 
\begin{enumerate}
\item $$\rho(Q)(v)(w) = \rho(Q)(v \odot w) =0.$$ 
\item $$\rho(Q)(\xi^{n}_{p} )(\xi^l_p)= \rho(Q)(\xi^{n}_{p} \odot \xi^l_p) = 0, \  \forall l \leq n-1.$$ 
\item If $Q \in  ker(\mu_{2n-2})$, and $Q \not\in ker(\mu_{2n})$, then 
$$\rho(Q)(\xi^{n}_p)(\xi^{n}_p)  = \rho(Q)(\xi^{n}_p \odot \xi^{n}_p) =  c_{n, 2n-2} \mu_{2n}(Q)(p) \neq 0.$$

\item For $n \geq 2$ the same result holds if $C$ is a general curve of genus $g$, $p \in C$ is any point and $n < \frac{\sqrt{g-2} }{4} -1$. 

\item For $n=1$ the same result  holds  for every curve $C$ of genus at least 5 which is non hyperelliptic and non trigonal and $\forall p \in C$. 
\end{enumerate}

\end{THEOR}
\begin{proof}
In the case $n=1$ we already know by \cite[Theorem 3.1]{cpt} that if $Q \in  ker(\mu_{0})$, and $Q \not\in ker(\mu_{2})$, then 
$\rho(Q)(\xi^{1}_p)(\xi^{1}_p)  = \rho(Q)(\xi^{1}_p \odot \xi^{1}_p) =  c_{1, 0} \mu_{2}(Q)(p) \neq 0$ for $p$ general. 
Moreover, in \cite[Theorem 6.1]{cf1} it is proven that if $C$ is a smooth curve of genus $g \geq  5$, that is non hyperelliptic
and non trigonal, then, for any $p \in C$, there exists a quadric $Q \in I_2(\omega_C)$ such that $\mu_2(Q)(p) \neq 0$. So the above result for $n=1$ holds  for every curve $C$ of genus at least 5 which is non hyperelliptic and non trigonal and $\forall p \in C$, which shows $(5)$.

Assume $n \geq 2$, then by Proposition \ref{mainprop} we have $ \rho(Q)(\xi^k_p)(\xi^l_p) = 0$, if $1\leq l+k \leq 2n-1, $ so if $v, w \in \langle \xi^1_p,..., \xi^{n-1}_p \rangle$, $\rho(Q)(v)(w) = 0$, $\rho(Q)(\xi^{n}_{p} )(\xi^l_p) = 0$, $ \forall l \leq n-1$, so $(1)$ and $(2)$ hold. Moreover $\rho(Q)(\xi^{n}_p)(\xi^{n}_p) = c_{n,2n-2} \mu_{2n}(Q)(p) $. Then if $p$ is a general point, $\rho(Q)(\xi^{n}_p)(\xi^{n}_p) = c_{n,2n-2} \mu_{2n}(Q)(p)  \neq 0$, since $Q \not\in ker(\mu_{2n})$, so $(3)$ holds. 

To prove $(4)$, first notice that if $n < \frac{\sqrt{g-2} }{4} -1$, then $n < 2g-2$, so hypotheses $(*)$ are satisfied. In  \cite{ro} it is shown that for any $k < \frac{\sqrt{g-2} }{2} -2$, the map $\mu_k$ is surjective, hence its image is $H^0((k+2)\omega_C)$ which is base point free (this generalises a result obtained in \cite{ccm}, saying that for a general curve  $C$ of genus $\geq 18$, the map $\mu_2$ is surjective). Therefore if $n < \frac{\sqrt{g-2} }{4} -1 < 2g-2$,  $\forall p \in C$, there exists $Q \in  ker(\mu_{2n-2})$ such that $\mu_{2n}(Q)(p)  \neq 0$. 

\end{proof}

Finally we show the following 

\begin{THEOR}
\label{mui}
Assume that for a curve $C$ of genus $g \geq 4$ and for some $n \leq 3g-3$, we have:

$ker(\mu_{2n}) \subsetneq ker(\mu_{2n-2}) \subsetneq  ker(\mu_{2n-4}) .... \subsetneq ker(\mu_{2})  \subsetneq ker(\mu_{0}) = I_2(K_C). $

Let $Y$ be a  germ of  a totally geodesic submanifold of ${\mathcal A}_g$ generically contained in ${\mathcal M}_g$ passing through $j(C)$. Then, for a generic $p \in C$, we have: 
\begin{enumerate}

\item $T_{(jC)}Y \cap \langle \xi^1_p, ..., \xi^n_p \rangle = \{0\}$. 

\item $\dim Y \leq 3g-3-n$. 
\end{enumerate}
\end{THEOR}
\begin{proof}
For every $i =0,...,2n-2$ choose a quadric $Q_i \in Ker (\mu_{2i})$ such that $Q_i \not\in ker(\mu_{2i+2})$ and denote by $U: = C \setminus \cup_{i=0}^{2n-2} Z(\mu_{2i}(Q_i))$, where $Z(\mu_{2i}(Q_i))$ denotes the zero set of the section  $\mu_{2i}(Q_i) \in H^0((2i+2)\omega_C)$. 
Take a  point $ p\in U$.  Assume $Y$ is a germ of totally geodesic submanifold of ${\mathcal A}_g$ generically contained in ${\mathcal M}_g$ passing through $j(C)$. Then $W := T_{j(C)}Y \subset H^1(C, T_C)$.  Denote by $V := \langle \xi^1_p, ..., \xi^n_p \rangle$, take $ v \in V \cap W$, and assume $v \neq 0$. Then $v = a_1 \xi^1_p + ...+ a_{n} \xi^{n}_p$. Let $1 \leq k \leq n$ be the maximum integer such that $a_{k} \neq 0$, then $v \in \langle \xi^1_p,...,\xi^{k}_p \rangle$. Then by Theorem \ref{thmA}, we have: 
$$\rho(Q_{k-1})(v \odot v)   = \sum^k_{i,j=1}a_i a_j \rho(Q_{k-1}) (\xi^i_p \odot \xi^j_p) =$$
$$= a^2_k\rho(Q_{k-1})(\xi^{k}_p \odot \xi^{k}_p)=  a^2_k c_{k,2k-2} \cdot \mu_{2k}(Q_{k-1})(p) \neq 0,$$ since $Q_{k-1} \not\in ker(\mu_{2k})$ and $p \in U$.  This is a contradiction, since $Y$ is totally geodesic, so $\rho(Q_{k-1})(v \odot v) =0$.    Hence $V \cap W = (0)$ and we have 
$$\dim (V +W) = \dim(V) + \dim(W) \leq \dim H^1(C, T_C) = 3g-3, $$
so $\dim(Y)= \dim (W) \leq 3g-3-n$. 

\end{proof}

\begin{REM}
\label{R-O}
Notice the the assumptions of Theorem \ref{mui} can be written as follows: $\mu_{2l} \neq 0$, $\forall l =0,...,n$. 
By \cite[Theorem D]{ro}, these assumptions are satisfied for the general curve of genus $g$,  if $n < \frac{\sqrt{g-2} }{4} -1$.

\end{REM}

We have the following
\begin{COR}
\label{bicanonical1}
\begin{enumerate}

\item Assume $C$ satisfies the assumptions of Theorem \ref{mui} and let $Y$ be a  germ of  a totally geodesic submanifold of ${\mathcal A}_g$ generically contained in ${\mathcal M}_g$ passing through $j(C)$.  Then there exists a finite set $S \subset C$ such that for every $p \in U = C \setminus S$, the projective tangent space ${\mathbb P}(T_{j(C)}Y )$ does not intersect the $(n-1)^{th}$ osculating plane of the bicanonical curve at $\phi_{|2K_C|}(p)$.
\item For the general curve of genus $g$, if $n < \frac{\sqrt{g-2} }{4} -1$, then  for any $p \in C$, ${\mathbb P}(T_{j(C)}Y )$ does not intersect the $(n-1)^{th}$ osculating plane of the bicanonical curve at $\phi_{|2K_C|}(p)$. 

\end{enumerate}

\end{COR}
\begin{proof}
$(1)$ follows immediately from the proof of Theorem \ref{mui}, and Remark \ref{bicanonical}. 
To prove $(2)$, first notice that $n < \frac{\sqrt{g-2} }{4} -1< 2g-2$, so assumption $(*)$ holds and $\forall p \in C$ the subspace ${\mathbb P}(\langle \xi_p^1,...,\xi_p^n\rangle )$ is intrinsically defined. If $C$ is general, by  \cite[Theorem D]{ro} we know that $\mu_{2l}$ is surjective  $\forall l =0,...,n$, hence $Im(\mu_{2l})$ is base point free.  So, $\forall i = 0,..., 2n-2$ and $\forall p \in C$, there exists a quadric $Q_i \in Ker (\mu_{2i})$ such that $\mu_{2i+2}(Q_i)(p)  \neq 0$. Then the proof of Theorem \ref{mui} applies verbatim, choosing $U =C$. 

\end{proof}

\begin{REM}
By Theorem \ref{thmA}, since  for a generic point $p \in C$  a basis for $H^1(C, T_C)$ is given by $\xi^1_p,...,\xi^{3g-3}_p$, if there exists a curve $C$ of genus $g \geq 4$ such that 

$$ker(\mu_{6g-6}) \subsetneq ker(\mu_{6g-8}) \subsetneq  ker(\mu_{6g-10}) .... \subsetneq ker(\mu_{2})  \subsetneq ker(\mu_{0}) = I_2(K_C),$$
then there do not exist germs of totally geodesic submanifolds of ${\mathcal A}_g$ generically contained in ${\mathcal M}_g$ passing through $j(C)$. 

Nevertheless, we aspect that such curves do not exist.

\end{REM}

Now we prove a result that only concerns higher Gaussian maps. 

\begin{THEOR}
\label{rank}
For every non hyperelliptic  curve $C$ of genus $g \geq 4$, we have: 
\begin{enumerate}
\item $ker (\mu_{6g-6}) =0$, hence $\mu_{2k} \equiv 0$, $\forall k > 3g-3$. 
\item $\forall 0 \leq l \leq 3g-3$, $\dim (ker (\mu_{6g-6-2(l+1)})) \leq (l+1)^2$, hence $rank(\mu_{6g-6-2l} ) \leq (l+1)^2$. 
\end{enumerate}

\end{THEOR}
\begin{proof}
If $Q \in ker(\mu_{6g-6})$, then by Proposition \ref{mainprop}, for $p$ general, we have $\rho(Q)(\xi_p^n)(\xi_p^l) =0$, $\forall n,l \leq 3g-3$. Since $\{\xi_p^1,...,\xi_p^{3g-3}\}$ is a basis of $H^1(C, T_C)$, this implies that $\rho(Q) \equiv 0$. 
Hence $Q =0$, since $\rho$ is injective (\cite[Corollary 3.4]{cf-trans}). This shows $(1)$. 

For $(2)$, let us first work out the case $l=0$. By Proposition \ref{mainprop}, if $Q \in ker(\mu_{6g-8})$, the matrix representing the quadric $\rho(Q)$ in the basis $\xi_p^1,...,\xi_p^{3g-3}$, where $p$ is general, is the following: 
$$ {\small \left( \begin{array}{cccccc}
0&.&.&.&0&0\\
.&.&.&.&.&.\\
.&.&.&.&.&.\\
.&.&.&.&.&.\\
0&.&.&.&0&0\\
0&.&.&.&0&a\\
\end{array} \right)}$$
where $a = \rho(Q)(\xi_p^{3g-3} \odot \xi_p^{3g-3} )$. So the restriction of $\rho$ to $ker(\mu_{6g-8})$ has rank at most one, and since $\rho$ is injective, $\dim  ker(\mu_{6g-8}) \leq 1$. So, if $ker(\mu_{6g-8})=0$, then $\mu_{6g-6} \equiv 0$. If  $ker(\mu_{6g-8})$ is one dimensional, then $\mu_{6g-6}$  is injective of rank 1.

In general, as above, using Proposition \ref{mainprop}, $\forall Q \in ker( \mu_{6g-6-2(l+1)}) $, and $p$ a general point, we have 
$$\rho(Q) (\xi_p^n \odot \xi_p^k) = 0, \  \forall n,k \leq 3g-4-l, \ \rho(Q) (\xi_p^{3g-3-l} \odot \xi_p^k) = 0, \  \forall k \leq 3g-4-l,$$
$$\rho(Q)(\xi_p^{3g-l+t} \odot \xi_p^k) = 0, \  \forall k \leq 3g-7-l-t, \ \forall t \leq l-3.$$

So, looking at the matrix representing $\rho(Q)$ in the basis $\{\xi_p^1,..., \xi_p^{3g-3}\}$, one immediately sees that the restriction of $\rho$ to $ker (\mu_{6g-6-2(l+1)})$ has rank at most $\sum_{k=0}^l(2k+1) = (l+1)^2$.  Hence by the injectivity of $\rho$, $\dim(ker (\mu_{6g-6-2(l+1)} )) \leq (l+1)^2$, and thus $rank(\mu_{6g-6-2l} ) \leq (l+1)^2$. This concludes the proof. 
\end{proof}

\begin{REM}
Notice that $\mu_{6g-6-2l} : ker (\mu_{6g-8-2l}) \ra H^0(C, \omega_C^{\otimes (6g-4-2l)})$. So the estimate in $(2) $ of Theorem \ref{rank} gives information only when $(l+1)^2 \leq h^0(C, \omega_C^{\otimes (6g-4-2l)}) = (g-1)(12g-9-4l)$, so $l \leq 1-2g + \sqrt{(g-1)(16g-9)} \leq 2g-2$, hence $6g-6-2l \geq 2g-2$. 
\end{REM}



\end{document}